\documentclass{amsart}

\usepackage{graphicx}%
\usepackage{multirow}%
\usepackage{amsmath,amssymb,amsfonts}%
\usepackage{amsthm}%
\usepackage{mathrsfs}%
\usepackage[title]{appendix}%
\usepackage{xcolor}%
\usepackage{textcomp}%
\usepackage{manyfoot}%
\usepackage{booktabs}%
\usepackage{algorithm}%
\usepackage{algorithmicx}%
\usepackage{algpseudocode}%
\usepackage{listings}%
\usepackage{longtable}
\usepackage{nicematrix,tikz}

\newtheorem{theorem}{Theorem}[section]
\newtheorem{lemma}[theorem]{Lemma}
\newtheorem{corollary}[theorem]{Corollary}

\theoremstyle{definition}
\newtheorem{definition}[theorem]{Definition}
\newtheorem{example}[theorem]{Example}

\theoremstyle{remark}

\newtheorem{conjecture}[theorem]{Conjecture}

\numberwithin{equation}{section}

\newcommand{\tr}{\operatorname{tr}}
\newcommand{\rank}{\operatorname{rank}}
\newcommand{\diag}{\operatorname{diag}}

\newcommand\BLUE{\color{blue}}

\begin{document}

\title[Factorization of a prime matrix]{Factorization of a prime matrix in even blocks}


\author[Haoming Wang]{Haoming Wang}
\address{School of Mathematics, Sun Yat-sen University, Guangzhou {\rm510275}, China}
\email{wanghm37@mail2.sysu.edu.cn}
\thanks{}


\date{}

\dedicatory{}

\begin{abstract} In this paper, a matrix is said to be prime if the row and column of this matrix are both prime numbers. We establish various necessary and sufficient conditions for developing matrices into the sum of tensor products of prime matrices. For example, if the diagonal of a matrix blocked evenly are pairwise commutative, it yields such a decomposition. The computational complexity of multiplication of these algorithms is shown to be $O(n^{5/2})$. In the section 5, a decomposition is proved to hold if and only if every even natural number greater than 2 is the sum of two prime numbers.
\end{abstract}

\maketitle
\tableofcontents

\section{Introduction}
In this note, a prime matrix is a $p \times q$ matrix with both $p$ and $q$ prime numbers. Assume $m$ and $n$ are natural numbers with the prime decomposition. The aim of this paper is two-fold:
\begin{itemize}
	\item Reducing an $m\times n$ matrix into the sum of tensor products of prime matrices.
	\item Reducing an $n\times n$ Hermite matrix into the sum of tensor products of prime matrices.
\end{itemize}

An $m \times n$ matrix might be developed into the sum of tensor products of prime matrices, 
In linear algebra, an $m \times n$ matrix $A$ has the spectral decomposition in the sense that
\begin{equation}
	A = \sum_{i=1}^{\min\{m,n\}}\sigma_{i} u_{i}v_{i}^{*},
\end{equation}
where $U = [u_{1},u_{2},\dots,u_{m}]$ and $V = [v_{1},v_{2}, \dots,v_{n}]$ are orthogonal matrices. It's quite hard to compute the $\sigma_{i}, u_{i}$ and $v_{i}$ when the size of matrix is large. The usual method relies on the computation of eigenvalues and eigenvectors of $AA^{*}$ and $A^{*}A$. If $n=m$, with the aid of usual elimination algorithm, the computational complexity is $O(n^{3})$, which increases the same as a cubic polynomial as the size of matrix gets large. However, this could be remedied provided a tensor form exists by the prime number decomposition theorem. 

If $m = p_{1}^{r_{1}} p_{2}^{r_{2}} \dots p_{k}^{r_{k}}$ and $n = q_{1}^{s_{1}} q_{2}^{s_{2}} \dots q_{l}^{s_{l}}$ are the prime decomposition of $m$ and $n$ respectively, the $m \times n$ matrix $A$ could be blocked according to the prime decomposition.  For example, an $m^{\prime} p \times n^{\prime}  q$ matrix $A$ may be blocked as follows:
\begin{equation}
		A = \begin{bNiceMatrix}[first-row,last-col]
		q & \dots & q &\\ 
		A_{11} & \dots & A_{1n^{\prime}} & p\\
		\vdots & \ddots & \vdots & \vdots \\
		A_{m^{\prime}1} & \dots & A_{m^{\prime}n^{\prime}}& p\\
	\end{bNiceMatrix}
	\label{eq: block matrix 1}
\end{equation}
where each $A_{ij}$ has $p$ rows and $q$ columns. If the diagonal set of block sub-matrices $\{A_{ii}\}$ consists of a pairwise commutative family of matrices of $p$ rows and $q$ columns in the sense that the following equation is satisfied
\begin{equation}	
	AB^{*} = BA^{*}, \quad  A^{*}B = B^{*}A,
	\label{eq: commutativity}
\end{equation}
for each $A$ and $B$ in the family of matrices, then there exists a set of orthonormal column vectors $\{b_{i}\}$ and $\{c_{j}\}$ which diagonalizes each matrix $A_{ii}$, say,
\begin{equation}
	A_{ii} = \sigma_{ii}^{k} b_{k}c_{k}^{*},
\end{equation}
where the sum is taken over repeated index. Therefore, the $m^{\prime} p \times n^{\prime}  q$ block matrix can be developed into the sum of tensor products
\begin{equation}
(1_{m^{\prime } \times m^{\prime }} \otimes B)^{*}	 A (1_{n^{\prime } \times n^{\prime }} \otimes C) =  \begin{bNiceMatrix}[first-row,last-col]
		q & \dots & q & \\ 
		\diag(\sigma_{11}^{k}) & \dots & A^{\prime}_{1n^{\prime}} & p\\
		\vdots & \ddots & \vdots & \vdots \\
		A^{\prime}_{m^{\prime}1} & \dots & \diag(\sigma_{m^{\prime}n^{\prime}}^{k}) & p\\
	\end{bNiceMatrix}
	\label{eq: block matrix 2}
\end{equation}
where $1_{m\times n}$ denote the $m\times n $ matrix with entries all $1$, $B = [b_{1},b_{2},\dots,b_{p}]$, $C = [c_{1},c_{2}, \dots,c_{q}]$ and $A_{ij}^{\prime}(k,l) = b^{*}_{k} A_{ij} c_{l}$. The symbol $\otimes$ means the tensor product defined for an $m \times n$ matrix $A$ and $p \times q$ matrix $B$: 
\begin{equation}
	A \otimes B = \begin{bNiceMatrix}
		a_{11}B & a_{12}B & \dots & a_{1n}B \\
		a_{21}B & a_{22}B & \dots & a_{2n}B \\
		\vdots &  & \ddots &  \cdots \\
		a_{m1}B & a_{m2}B & \dots & a_{mn}B 
	\end{bNiceMatrix}
\end{equation}
is defined to be a $pm\times qn$ block matrix. We will defer the non-square problem to the section 6.

The second aim of this paper is to reduce an $n \times n$ Hermite matrix $A$ where $A = A^{*}$ into the sum of tensor products in the following four type.
\begin{itemize}  
	\item[$T_{1}$] The matrix $A$ blocked according to (\ref{eq: block matrix 1}) can be developed into the sum
	$$\displaystyle A = \sum_{i,j} B_{ij} \otimes C_{ij},$$ where $B_{ij}$ and $C_{ij}$ are $n^{\prime} \times n^{\prime}$ and  $q \times q$ matrices respectively.  
	\item[$T_{1\frac{1}{2}}$] The matrix $A$ blocked according to (\ref{eq: block matrix 1}) can be developed into the sum
	$$\displaystyle A = \sum_{i} B_{i} \otimes C_{i},$$ where $B_{i}$ and $C_{i}$ are $n^{\prime} \times n^{\prime}$ and  $q \times q$ matrices respectively. 
	\item[$T_{2}$] The matrix $A$ blocked according to (\ref{eq: block matrix 1}) can be developed into the sum
	$$\displaystyle A = \sum_{i,j} \alpha_{ij} B_{i} \otimes C_{j},$$ where $\alpha_{ij}$, $B_{i}$ and $C_{j}$ are  real numbers, $n^{\prime} \times n^{\prime}$ and  $q \times q$  matrices respectively.
	\item[$T_{3}$] The matrix $A$ blocked according to (\ref{eq: block matrix 1}) can be developed into the sum
	$$\displaystyle A =  B \otimes C,$$ where $B$ and $C$ are $n^{\prime} \times n^{\prime}$ and  $q \times q$  matrices respectively.
\end{itemize}
where the required properties for $\alpha_{ij}, B_{ij},B_{i}, B, C_{ij},C_{i}, C$ will be explained 
in the section 2 explicitly. Here we only illustrate the main idea of this procedure. 
\begin{itemize}
	\item First, examine whether the diagonal set of block matrix $A$ in (\ref{eq: block matrix 1}) is pairwise commutative:
	\[A_{ii}A_{jj} = A_{jj}A_{ii},\quad i,j=1,2,\dots, n^{\prime},\]
	This requires $O(n^{5/2})$ computational complexity.
	\item Second, examine whether the lower-triangular set of block matrix $A$ in (\ref{eq: block matrix 1}) is pairwise commutative. 
	\[A_{ij}A_{i^{\prime}j^{\prime}} = A_{i^{\prime}j^{\prime}}A_{ij},\quad i = j,\dots,n^{\prime}, i^{\prime}=j^{\prime},\dots, n^{\prime}, j,j^{\prime} = 1,\dots, n^{\prime}\]
	This requires $O(n^{5/2})$ computational complexity.
	\item Third, examine whether the diagonal set of the family of $n^{\prime} \times n^{\prime}$ matrices $\{B_{i}\}$ in (\ref{eq: block matrix 2}) where $B_{i}(k,l) = b^{*}_{i} A_{kl} c_{i}$ is pairwise commutative.
	\[B_{i}B_{j} = B_{j}B_{i},\quad i,j=1,2,\dots, n^{\prime},\] 
	This requires $O(n^{5/2})$ computational complexity.
	\item Finally, compute the rank of $n^{\prime} \times n^{\prime}$ real matrix $(\alpha_{ij})$  where $\alpha_{ij} = b_{j}B_{ii}b_{j}^{*}$ via the similar diagonal form:
	\[P_{1}P_{2}\dots P_{n}\begin{bmatrix}
		1 &  & & O \\
		 & 1 & & \\
		 &  &\ddots & \\
		O & & &
	\end{bmatrix}Q_{n}^{-1}\dots Q_{2}^{-1}Q_{1}^{-1}
	\]
	This requires $O(n^{3/2})$ computational complexity.
\end{itemize}
The total computational complexity of these algorithm is therefore $O(n^{5/2})$, much less than that of the usual matrix multiplication, which is known to be $O(n^{3})$. For this algorithm requires, in the worst case, $n^{3}$ multiplications of scalars and $n^{3}-n^{2}$ additions for computing the product of two $n \times n$ matrices. This enables us to utilize the above algorithm as a first examination so that if passed, it requires at most 
$n^{5/2} + n^{2} - n^{3/2}$ multiplications of scalars and
$n^{5/2} - n^{3/2}$ additions for multiplying two $n\times n$ matrices. These methods are introduced in section 4 and their algorithms are illustrated in Appendix B.

However, we are only interested in the case $n$ is divisible, that is, $n$ contains at least two prime factors that may not be distinct since if the order $n$ of the square matrix $A$ is a prime number, this problem has only trivial solution,. According to the fundamental theorem of arithmetic, every natural number has a unique decomposition as the product of prime numbers. The proof of arithmetic fundamental theorem is inductive so we use the induction to derive our desired decomposition too. We begin to start with a decomposable number $n =a b$ where $a$ and $b$ are two natural numbers larger than 2. 
 
\section{The case $n = ab$}
Here we give some assumptions about the reduction of an square matrix $A$ of order $n=ab$ into the sum of tensor products. Familiar with them is helpful in our coming discussion. Let's look at the Hermite case.
\begin{definition} Let $A$ be an Hermite matrix of order $ab$. Block $A$ as
	
	\begin{equation}
		A = \begin{bmatrix}
			A_{11} & A_{12} & \dots & A_{1a}\\
			A_{21} & A_{22} & \dots & A_{2a}\\
			\vdots &  & \ddots & \vdots \\
			A_{a1} & A_{a2} & \dots & A_{aa}
		\end{bmatrix}
		\label{eq: block matrix}
	\end{equation}
	where each $A_{ij}$ is an Hermite matrix of order $b$. We assume $A$ is decomposable of the type
	\begin{itemize}
		\item[$T_{1}$] {\it with respect to some column bases} if there exist orthonormal column vectors $\{c_{j}\}$ of order $b$ such that $A$ can be developed into the sum
		\[A = \sum_{i,j} B_{ij} \otimes C_{ij}\]
		where $C_{ij} = c_{i} c_{j}^{*}$ and $B_{ij} = \left(c_{i}^{*} A_{kl} c_{j}\right)$. In this case, $ B_{ij}(k,k)  = \beta_{i}(k)\delta_{ij}$. In a word, 
		$B_{ij}$ vanishing on diagonal if $i\neq j$.
		\item[$T_{1\frac{1}{2}}$] {\it with respect to some mutually orthogonal row coefficients} if it is $T_{1}$ decomposable and additionally, $B_{ij} = B_{ii}\delta_{ij}$, i.e. 
		$B_{ij}$ vanishing if $i\neq j$.
		\item[$T_{2}$] {\it with respect to some row-column bases} if it is $T_{1}$ decomposable and there exist further orthonormal column vectors $\{b_{i}\}$ of order $a$ such that $A$ can be developed into the sum
		\[A = \sum_{i,j} \alpha_{ij} B_{i} \otimes C_{j},\]
		where $B_{i} = b_{i}b_{i}^{*}$ and $C_{j} = c_{j}c_{j}^{*}$.
		\item[$T_{3}$] {\it totally decomposable with respect to some row-column bases} if it is $T_{2}$ decomposable and additionally, there exist $\beta_{i}, \gamma_{j}$ such that $\alpha_{ij} = \beta_{i} \gamma_{j}$ and
		\[A =  B \otimes C \]
		where $B = \sum_{i} \beta_{i} b_{i} b_{i}^{*}$ and $C = \sum_{j} \gamma_{j} c_{j}c_{j}^{*}$.
	\end{itemize}
\end{definition}

As for the non-Hermite case, we will assume $A$ has $n$ distinct eigenvalues. In this setting, replace the word orthogonal with independent and rephrase word for word in the above paragraph. We leave this work to the reader.

\section{The canonical diagonal form}

We define a binary operation between an $m \times n$ matrix $A$ and a $k\times l$ matrix $B$ when $n$ and $k$ has common divisor $a$ as follows: block $A = [A_{1}, A_{2}, \dots, A_{a}]$ and $B = [B^{*}_{1}, B^{*}_{2}, \dots, B^{*}_{a}]^{*}$ according to the partition of $n = as$ columns of $A$ and $k = at$ rows of $B$ and $m\times s$ sub-matrices $A_{i}$ and $t \times l$ sub-matrices $B_{i}$, $i=1,2,\dots,a$. 
\begin{equation}
	A \otimes_{a,a} B = \sum_{i=1}^{a} A_{i} \otimes B_{i}
	\label{eq: new tensor}
\end{equation}
The result $A \otimes_{a,a}B$ is a $tm\times ls$ matrix, dividing the factor of $km\times ln$ both by $a$. 

Besides, we define a unary operation on an $m \times n$ matrix $A$ when $m$ and $n$ has both common divisor $a$: block $A$ according to $m=ar$ and $n=as$
\begin{equation}
	\begin{bmatrix}
		A_{11} & A_{12} & \dots & A_{1s}\\
		A_{21} & A_{22} & \dots & A_{2s}\\
		\vdots &  & \ddots & \vdots \\
		A_{r1} & A_{r2} & \dots & A_{rs}
	\end{bmatrix}
\end{equation}
where $A_{ij}$s are $a \times a$ matrices and then define
\begin{equation}
	\underset{a,a}{\bigoplus} A = \sum_{i,j} A_{ij}
	\label{eq: new sum}
\end{equation}

Now return to our topic of decomposing an Hermite matrix of order $n =ab$. We will denote the block matrix $C = (C_{ij})$ according to $C_{ij}$ as described in section 2 if it appears.
\begin{figure}[h!]
	\centering
	\begin{minipage}{0.49\linewidth}
		\centering
		\[D_{1} = \begin{bmatrix}
			B_{11} & B_{12} & \dots & B_{1a}\\
			B_{21} & B_{22} & \dots & B_{2a}\\
			\vdots &  & \ddots & \vdots \\
			B_{a1} & B_{a2} & \dots & B_{aa}
		\end{bmatrix}\]	
		\caption{The canonical diagonal form with $A = \oplus_{a,a} \left(D_{1} \otimes_{a,a} C \right)$}
		\label{fig: T_{1}}
	\end{minipage}
	\begin{minipage}{0.49\linewidth}
		\centering
		\[D_{1\frac{1}{2}} = \begin{bmatrix}
			B_{11} & \dots & & O\\
			& B_{22} &  & \vdots\\
			\vdots &  & \ddots & \\
			O &  & \dots & B_{aa}
		\end{bmatrix}\]
		\caption{The canonical diagonal form with $A = \oplus_{a,a} \left(D_{1\frac{1}{2}} \otimes_{a,a}  C\right)$}
		\label{fig: T_{1}1/2}
	\end{minipage}
\end{figure}		
\begin{figure}[h!]
	\centering
	\begin{minipage}{0.49\linewidth}
		\centering
		\[D_{2} = \begin{bmatrix}
			\alpha_{11} & \alpha_{12} & \dots & \alpha_{1a}\\
			\alpha_{21} & \alpha_{22} & \dots & \alpha_{2a}\\
			\vdots &  & \ddots & \vdots\\
			\alpha_{a1} & \alpha_{a2} & \dots & \alpha_{aa}
		\end{bmatrix}\]	
		\caption{The canonical diagonal form with $A = D_{2} \cdot (B\otimes C)$.}
		\label{fig: T_{2}}
	\end{minipage}
	\begin{minipage}{0.49\linewidth}
		\centering
		\[D_{3} = \begin{bmatrix}
			\beta_{1}\gamma_{1} & \beta_{1}\gamma_{2} & \dots & \beta_{1}\gamma_{a}\\
			\beta_{2}\gamma_{1} & \beta_{2}\gamma_{2} & \dots & \beta_{2}\gamma_{a}\\
			\vdots &  & \ddots & \vdots\\
			\beta_{a}\gamma_{1} & \beta_{a}\gamma_{2} & \dots & \beta_{a}\gamma_{a}
		\end{bmatrix}\]	
		\caption{The canonical diagonal form with $A = D_{3} \cdot (B\otimes C)$.}
		\label{fig: T_{3}}
	\end{minipage}
\end{figure}

Since an Hermite matrix may have a canonical diagonal decomposition, we tabulate some equivalent conditions for them to hold explicitly in Table (\ref{tab: equivalent conditions}):
\begin{longtable}{ccl}
	\caption{Comparison of $T_{1},T_{1\frac{1}{2}},T_{2}$ and $T_{3}$.}\\
		\hline
		$T_{1}$ & $\Leftrightarrow$  & $A = \sum_{i,j} B_{ij} \otimes C_{ij}$ where $C_{ij}=c_{i}c_{j}^{*}$.\\
		&  $\Leftrightarrow$   & $C_{ij} = c_{i}c_{j}^{*}$ consists of common eigenvectors $c_{i}$ of $A_{kk}$, independent of \\
		& & $k$, corresponding to eigenvalues $B_{ii}(k,k)$.\\	
		\hline
		$T_{1\frac{1}{2}}$ & $\Leftrightarrow$ & $A = \sum_{i} B_{i} \otimes C_{i}$ where $C_{i} = c_{i}c_{i}^{*}$.\\
		&  $\Leftrightarrow$   & $C_{i} = c_{i}c_{i}^{*}$ consists of common eigenvectors $c_{i}$ of $A_{kl}$, independent of \\
		& & $k,l$, corresponding to eigenvalues $B_{i}(k,l)$.\\	
		\hline 
		$T_{2}$ & $\Leftrightarrow$ & $A = \sum_{i,j} \alpha_{ij} B_{i} \otimes C_{j}$ where $B_{i} = b_{i}b_{i}^{*}$ and $C_{j} = c_{j}c_{j}^{*}$\\
		& $\Leftrightarrow$ & $B_{i} = b_{i}b_{i}^{*}$, $C_{j} = c_{j}c_{j}^{*}$ consists of eigenvectors $b_{i}\otimes c_{j}$ of $A$,\\
		& & corresponding to eigenvalues $\alpha_{ij}$.\\
		\hline
		$T_{3}$ & $\Leftrightarrow$ & $A= B \otimes C$ and $B = \sum_{i} \beta_{i} b_{i}b_{i}^{*}$, $C = \sum_{j} \gamma_{j} c_{j} c^{*}_{j}$.\\ 
		& $\Leftrightarrow$ & $b_{i}\otimes c_{j}$ are eigenvectors of $A$, corresponding to eigenvalues $\beta_{i}\gamma_{j}$.\\
		\hline		
		\label{tab: equivalent conditions}
\end{longtable}

From Table (\ref{tab: equivalent conditions}), there are some obvious deductions.

\begin{theorem} Let $A$ be an $n \times n$ Hermite matrix blocked in the form (\ref{eq: block matrix}).
	\begin{itemize}
		\item[1.] $T_{3} \Rightarrow T_{2} \Rightarrow T_{1\frac{1}{2}} \Rightarrow T_{1}$.
		\item[2.] The converses $T_{1} \Rightarrow T_{1\frac{1}{2}}$ and $T_{1} \Rightarrow T_{2}$ and $T_{2} \Rightarrow T{3}$ are already made clear in the definition.
		\item[3.] If the diagonal set $\{A_{ii}\}$ is pairwise commutative, then $T_{1}$ holds.
		\item[4.] If the lower-triangular set $\{A_{ij}\}$ is pairwise commutative, then $T_{1\frac{1}{2}}$ holds. 
		\item[5.] $T_{1\frac{1}{2}}$ + $b_{j}$ is the eigenvector of $B_{i}$, independent of $i$, corresponding to eigenvalue $\alpha_{ij}$ $\Rightarrow T_{2}$.
		\item[6.] $T_{2}$ + $\alpha_{ij}=\beta_{i} \gamma_{j} \Rightarrow T_{3}$;
		\item[7.] $T_{1\frac{1}{2}}$ + $B_{jj} = \omega_{j} B$ where $\omega_{j}$ is constant and $B$ Hermite matrix $\Rightarrow T_{3}$. 
	\end{itemize}
\end{theorem}

\begin{proof}
	We only prove 3 and 4. This is because any permutation $\sigma_{a}$ of $\{1,2,\dots,n\}$ can be decomposed into the product of the replacement $\sigma_{n} = \tau_{1}\tau_{2}\dots\tau_{m}$ and commutativity means simultaneous diagonalization.
\end{proof}

\section{The computational complexity}

Take the matrix multiplication for example, This algorithm requires, in the worst case, $n^{3}$ multiplications of scalars and $n^{3}-n^{2}$ additions for computing the product of two square $n \times n$ matrices. Its computational complexity is therefore $O(n^{3})$. If the $n^{2} \times n^{2}$ matrix blocked evenly in the form (\ref{eq: block matrix}), we execute the following algorithm:

\begin{itemize}
	\item First, examine whether the diagonal set of block matrix $A$ in 
	\[\begin{bNiceMatrix}
		\BLUE A_{11} & A_{12} & \dots & A_{1n}\\
		A_{21} & \BLUE A_{22} & \dots & A_{2n}\\
		\vdots & & \BLUE \ddots & \vdots \\
		A_{n1} & A_{n2} & \dots & \BLUE A_{nn}
		\CodeAfter
		\tikz \draw (1-|2) |- (2-|3) |- (3-|4) |- (4-|5);
		\tikz \draw (2-|1) -| (3-|2) -| (4-|3) -| (5-|4);
		\end{bNiceMatrix}\]
	is pairwise commutative:
	\[A_{ii}A_{jj} = A_{jj}A_{ii},\quad i,j=1,2,\dots, n,\]
	This step requires, in the worst case $\displaystyle  2{n \choose 2}\cdot n^{3} = n^{5} - n^{4}$ multiplication of scalars and $\displaystyle 2{n \choose 2} \cdot (n^{3} - n^{2}) = n^{5} - 2n^{4} +n^{3}$ additions.
	\item Second, examine whether the lower-triangular set of block matrix $A$ in 
	\[\begin{bNiceMatrix}
		\BLUE A_{11} & A_{12} & \dots & A_{1n}\\
		\BLUE A_{21} & \BLUE A_{22} & \dots & A_{2n}\\
		\BLUE \vdots & & \BLUE \ddots & \vdots \\
		\BLUE A_{n1} & \BLUE A_{n2} & \BLUE \dots & \BLUE A_{nn}
		\CodeAfter
		\tikz \draw (1-|2) |- (2-|3) |- (3-|4) |- (4-|5);
	\end{bNiceMatrix}\]
	is pairwise commutative. 
	\[A_{ij}A_{i^{\prime}j^{\prime}} = A_{i^{\prime}j^{\prime}}A_{ij},\quad i = j,\dots,n^{\prime}, i^{\prime}=j^{\prime},\dots, n^{\prime}, j,j^{\prime} = 1,\dots, n\]
	This step requires, in the worst case $\displaystyle  2{n^{2} \choose 2}\cdot n = n^{5} - n^{3}$ multiplication of scalars and $\displaystyle 2{n^{2} \choose 2}\cdot (n-1) = n^{5} - n^{4} - n^{3} +n^{2}$ additions.
	\item Third, examine whether the diagonal set of the family of $n \times n$ matrices $\{B_{ii}\}$ in 
	\[\begin{bNiceMatrix}
		\BLUE B_{11} & \dots &  & O\\
		& \BLUE B_{22} &  & \vdots\\
		\vdots & & \BLUE \ddots & \\
		O &  & \dots & \BLUE B_{nn}
		\CodeAfter
		\tikz \draw (1-|2) |- (2-|3) |- (3-|4) |- (4-|5);
		\tikz \draw (2-|1) -| (3-|2) -| (4-|3) -| (5-|4);
	\end{bNiceMatrix}\]
	where $B_{ii}(k,l) = b^{*}_{i} A_{kl} c_{i}$ is pairwise commutative.
	\[B_{ii}B_{jj} = B_{jj}B_{ii},\quad i,j=1,2,\dots, n,\] 
	This step requires, in the worst case $\displaystyle  2{n \choose 2}\cdot n^{3} = n^{5} - n^{4}$ multiplication of scalars and $\displaystyle 2{n \choose 2}\cdot (n^{3}-1^{2}) = n^{5} - 2n^{4} + n^{3}$ additions.
	\item Finally, examine whether the rank of $n  \times n $ real matrix $(\alpha_{ij})$ is 1
	\[\begin{bmatrix}
		{\BLUE \alpha_{11}} (= \beta_{1}\gamma_{1}) & {\BLUE \alpha_{12}} (= \beta_{1}\gamma_{2}) & \dots & {\BLUE \alpha_{1n}} (= \beta_{1}\gamma_{n})\\
		{\BLUE \alpha_{21}} (=  \beta_{2}\gamma_{1}) & {\BLUE \alpha_{22}} (= \beta_{2}\gamma_{2}) & \dots & {\BLUE \alpha_{2n}} (= \beta_{2}\gamma_{1})\\
		\vdots &  & \ddots & \vdots\\
		{\BLUE \alpha_{n1}} (= \beta_{n}\gamma_{1}) & {\BLUE \alpha_{n2}} (= \beta_{n}\gamma_{2}) & \dots & {\BLUE \alpha_{nn}} (= \beta_{n}\gamma_{n})
	\end{bmatrix}\]
	where $\alpha_{ij} = b_{j}B_{ii}b_{j}^{*}$. This could be done via the similar diagonal form:
	\[P_{1}P_{2}\dots P_{n^{\prime}}\begin{bmatrix}
		1 &  & & O \\
		&  1 & & \\
		&  &  \ddots & \\
		O & & &
	\end{bmatrix}Q_{n^{\prime}}^{-1}\dots Q_{2}^{-1}Q_{1}^{-1}
	\]
	This step requires, in the worst case $\displaystyle  2{n \choose 2}\cdot n = n^{3} - n^{2}$ multiplication of scalars and $\displaystyle 2{n \choose 2}\cdot (n-1) = n^{3} - 2n^{2} + n$ additions.
\end{itemize}

\begin{longtable}[htbp]{ccc}
	\caption{Computing multiplication of two square $n^{2} \times n^{2}$ matrices.}\\
		\hline
		& Multiplications & Additions \\
		\hline
		$T_{1}$ & $2A_{2}^{n} \cdot n^{3} + n^{n}$ & $n^{5} - 2n^{4} +n^{3}$\\
		\hline
		$T_{1\frac{1}{2}}$ & $2A^{n^{2}}_{2} \cdot n + n^{4}$ & $n^{5} - n^{4} - n^{3} +n^{2}$\\
		\hline
		$T_{2}$ & $2A^{n^{2}}_{2} \cdot n + n^{3}$ & $2n^{4} + n^{3}$\\
		\hline
		$T_{3}$ & $2A^{n}_{2} \cdot n + n^{3}$ & $n^{3} - 2n^{2} + n$\\
		\hline
	\label{tab: copmplexity}
\end{longtable}
The computational complexity is hence not surprisingly $O(n^{5/2})$, much less than usual multiplication.

\begin{theorem} If $A$ is an Hermite matrix of decomposable order $n = p_{1}^{r_{1}} p_{2}^{r_{2}}\dots  p_{s}^{r_{s}}$, then it has decomposable type $4^{r_{1} + \dots + r_{s}}$, according to Definition 2.1. The computational complexity of this algorithm is also $O(n^{5/2})$.	
\end{theorem}

\section{The semi-prime and qausi-prime matrix}

\begin{definition}[Semi-prime matrix] A semi-prime matrix is an $m\times n$ matrix where the minimum of $m$ and $n$ is a prime number.
\end{definition}

Recall an $m \times n$ matrix $A$ is said to have a left inverse if there exists an $m\times m$ matrix $B$ such that $BA = \begin{bNiceArray}{c|c}
	I_{m} & O
\end{bNiceArray}$ and a right inverse if there exists an $n \times n$ matrix $C$ such that $AC = \begin{bNiceArray}{c}
	I_{n} \\
	\hline
	O
\end{bNiceArray}$. It is said be invertible if both inverses exist and quasi-invertible if there exists only one-side inverse. 

\begin{definition}[Quasi-prime matrix] A matrix is said to be quasi-prime if it is the product of two quasi-invertible semi-prime matrices.
\end{definition}

Two $n \times n$ matrices $A$ and $B$ are said to be orthogonal if $AB^{*} = A^{*}B = O$.
 
\begin{conjecture}
	Every $n \times n$ matrix where $n$ is an even natural number greater than 2 is the sum of two orthogonal quasi-prime matrices. 
\end{conjecture}

\begin{lemma}There are unique orthogonal matrices $U$ and $V$ for any quasi-prime matrix $A$ such that $U^{*}AV$ is diagonal, and the number of whose non-zero entries is a prime number.
\end{lemma}
\begin{proof} Suppose there are two $A = U_{1}V_{1} = U_{2}V_{2}$ decomposition of the $m \times n $ matrix $A$ given by Definition 5.2, then by the uniqueness of spectral decomposition, the $U_{1},U_{2}$ and $V_{1}, V_{2}$ are unitarily equivalent matrices.
	If the $m \times n$ matrix $A$ has the spectral decomposition, 
	\[A = \sum_{i=1}^{\min\{m,n\}} \sigma^{2}_{i} u_{i}v_{i}^{*}\]
	where $U = [u_{1},u_{2},\dots,u_{m}]$ and $V = [v_{1},v_{2}, \dots,v_{n}]$ are orthogonal matrices, it suffices to prove that the number of non-zero $\sigma_{i}$ is a prime number. This follows from the unique decomposition of semi-prime matrix.
\end{proof}

\begin{corollary} Every quasi-invertible semi-prime matrix is quasi-prime.
\end{corollary}	
\begin{proof}
	Suppose $m$ is the minimum of $m$ and $n$. An $m \times n$ matrix that is semi-prime can be expressed as the product $A = I_{m}A$. 
\end{proof}
\begin{corollary}
	The product of two $m \times n$ and $n \times k$ quasi-prime matrices is a $m \times k$ quasi-prime matrix.
\end{corollary}
\begin{theorem}
	The conjecture 5.3 holds true if and only if every even natural number greater than 2 is the sum of two prime numbers.
\end{theorem}
\begin{proof} The necessity is easily seen from the preceding lemma and orthogonal relation.
	It suffices to prove the sufficiency. If an even natural number $n$ larger than 2 can be developed as the sum of two prime numbers $p$ and $q$, block the $n \times n$ matrix $A$ as follows
	\begin{equation}
		\begin{bNiceArray}{c|c}[first-row,last-col]
		    p & q &\\
			A_{11} & & p\\
		\hline
		 	& O & 	q
		\end{bNiceArray}
		+ \, 	\begin{bNiceArray}{c|c}[first-row,last-col]
			p & q &\\
		 	& A_{12} & p\\
			\hline
			O & & q
		\end{bNiceArray}
		 \, + \dots +\, \begin{bNiceArray}{c|c}[first-row,last-col]
		 	p & q &\\
		 	 O & & p \\
		 	\hline
		  & A_{22} & q 
		 \end{bNiceArray}
	\end{equation}
	and choose two $n \times n$ orthogonal matrices $U = [u_{1}, u_{2}, \dots, u_{n}]$ and $V = [v_{1}, v_{2}, \dots, v_{n}]$ such that
	\begin{equation}
	U^{*} A V = \begin{bNiceArray}{c|c}[first-row,last-col]
			p & q &\\
			\Lambda_{1} & O & p\\
		\hline
			O & \Lambda_{2} & q
		\end{bNiceArray}
	\end{equation}	
	where $\Lambda_{1} = \diag(\lambda_{1},\dots,\lambda_{p})$ and $\Lambda_{2} = \diag(\lambda_{p+1},\dots,\lambda_{n})$. Then block $U$ and $V$ as 
	\begin{eqnarray*}
		U = \begin{bNiceArray}{c|c}[first-row,last-col]
			p & q &\\
			U_{1} & U_{2} & n 
		\end{bNiceArray}\\
		V = \begin{bNiceArray}{c|c}[first-row,last-col]
			p & q & \\
			V_{1} & V_{2}& n 
		\end{bNiceArray}
	\end{eqnarray*}
	The spectral decomposition yields
	\[A = U_{1} \Lambda_{1} V^{*}_{1} + U_{2} \Lambda_{2} V^{*}_{2}\]
	Since $U_{1}, V_{1}, U_{2}, V_{2}$ and $\Lambda_{1}, \Lambda_{2}$ are $n\times p$, $n\times p$, $n\times q, n\times q$ semi-prime matrices and $p\times p, q \times q$ invertible prime matrices, it follows from $V_{1}^{*}V_{2} = U_{1}^{*}U_{2} = O$ that the conjecture holds true. 
	\end{proof}

Two $m\times n$ matrices $A$ and $B$ are said to be congruent if there are two invertible square matrices $U$ and $V$ of order $m$ and $n$ respectively such that $UA = BV$.

\begin{conjecture}
	For every natural number $k$, there are infinitely many prime numbers $n$ and $(n -2k) \times (n)$ semi-prime matrices that are not congruent to each other. In particular, the semi-prime matrices can be chosen quasi-invertible.
\end{conjecture}

and a weak form of above

\begin{conjecture}
	For every natural number $k$, there are infinitely many prime numbers $n$ and $(n-2k) \times n$ quasi-prime matrices that are not congruent to each other.
\end{conjecture}

\begin{theorem}	The conjecture 5.8 holds true if and only if there are infinitely many twin primes with gaps equal $k$; The conjecture 5.9 holds true if and only if there are infinitely many twin primes with gaps less or equal than $k$. 
\end{theorem}




	

\section{The non-square problem}

The non-square problem arises in real-life applications since in this case the matrix we take into consideration is non-Hermite. It's usually impossible to find a diagonal form but we will try another method to reduce the non-square matrix. For example, if $m = p_{1}^{r_{1}} p_{2}^{r_{2}} \dots p_{k}^{r_{k}}$ and $n = q_{1}^{s_{1}} q_{2}^{s_{2}} \dots q_{l}^{s_{l}}$ are the prime decomposition of $m$ and $n$, the $m \times n$ matrix $A$ could be blocked as an $m^{\prime} p_{1} \times n^{\prime}  q_{1}$ matrix, say,
\begin{equation}
	A = \begin{bNiceMatrix}[first-row,last-col]
		q_{1} & \dots & q_{1} &\\ 
		A_{11} & \dots & A_{1n^{\prime}} & p_{1}\\
		\vdots & \ddots & \vdots & \vdots \\
		A_{m^{\prime}1} & \dots & A_{m^{\prime}n^{\prime}}& p_{1}\\
	\end{bNiceMatrix}
	\label{eq: block matrix 3}
\end{equation}
where each $A_{ij}$ has $p_{1}$ rows and $q_{1}$ columns. If the diagonal set of block sub-matrices $\{A_{ii}\}$ consists of a pairwise commutative family matrices of $p_{1}$ rows and $q_{1}$ columns in the sense that this relation is fulfilled:
\begin{equation}	
	AB^{*} = BA^{*}, \quad  A^{*}B = B^{*}A,
	\label{eq: commutativity 2}
\end{equation}
for each $A$ and $B$ in the family of matrices, then there exists two $p_{1} \times p_{1} $ and $q_{1} \times q_{1}$ orthogonal matrices $C^{L}= [c^{L}_{1}, c^{L}_{2}, \dots,c^{L}_{p_{1}}]$ and $C^{R}= [c^{R}_{1}, c^{R}_{2}, \dots,c^{R}_{q_{1}}]$ which diagonalizes each matrix $A_{ii}$, say,
\begin{equation}
	A_{ii} = \sigma_{ii}^{k} c^{L}_{k}c_{k}^{R*},
\end{equation}
where the sum is taken over repeated index. Therefore, the $m^{\prime} p_{1} \times n^{\prime}  q_{1}$ block matrix can be developed into the sum of tensor products
\[A = \sum_{i,j} B_{ij} \otimes C_{ij}\]
where $C_{ij} = c^{L}_{i} c^{R*}_{j}$, $B_{ij} = \left(c^{L*}_{i} A_{kl} c^{R}_{j}\right)$ and $B_{ij}(k,k)  = \beta_{i}(k)\delta_{ij}$. In a word, $B_{ij}$ vanishing on diagonal if $i\neq j$. Thus the $m \times n$ matrix is said to be decomposable of the type

\begin{itemize}
	\item[$T_{1}$] {\it with respect to some column bases} if there exist orthonormal column vectors $\{c^{L}_{j}\}$ and $\{c^{R}_{j^{\prime}}\}$ of order $p_{1}$ and $q_{1}$ respectively such that $A$ can be developed into the sum
	\[A = \sum_{i,j} B_{ij} \otimes C_{ij}\]
	where $C_{ij} = c^{L}_{i} c^{R*}_{j}$, $B_{ij} = \left(c_{i}^{L*} A_{kl} c_{j}^{R}\right)$. In this case, $ B_{ij}(k,k)  = \beta_{i}(k)\delta_{ij}$. In a word, $B_{ij}$ vanishing on diagonal if $i\neq j$.
	\item[$T_{1\frac{1}{2}}$] {\it with respect to some mutually orthogonal row coefficients} if it is $T_{1}$ decomposable and additionally, $B_{ij} = B_{ii}\delta_{ij}$, i.e. $B_{ij}$ vanishing if $i\neq j$.
	\item[$T_{2}$] {\it with respect to some row-column bases} if it is $T_{1}$ decomposable and there exist further orthonormal column vectors $\{b^{L}_{i}\}$ and $\{b^{R}_{i^{\prime}}\}$ of order $m^{\prime}$ and $n^{\prime}$ respectively such that $A$ can be developed into the sum
	\[A = \sum_{i,j} \alpha_{ij} B_{i} \otimes C_{j},\]
	where $B_{i} = b^{L}_{i}b_{i}^{R*}$ and $C_{j} = c_{j}^{L}c_{j}^{R*}$.
	\item[$T_{3}$] {\it totally decomposable with respect to some row-column bases} if it is $T_{2}$ decomposable and additionally, there exist $\beta_{i}, \gamma_{j}$ such that $\alpha_{ij} = \beta_{i} \gamma_{j}$ and
	\[A =  B \otimes C \]
	where $B = \sum_{i} \beta_{i} b^{L}_{i} b_{i}^{R*}$ and $C = \sum_{j} \gamma_{j} c_{j}^{L}c_{j}^{R*}$.
\end{itemize}
Equivalent conditions for these assumptions are similar to Table \ref{tab: equivalent conditions} and Theorem 3.1 provided the commutativity relation replaced by (\ref{eq: commutativity 2}).

\section{Conclusion}

There are still possible representations of matrix $B_{ij}(k,l)$ in $T_{1}$. For example, $B^{\prime}_{ij} = \left(c_{k}^{*} A_{ij} c_{l}\right)$. We could develop the matrix $A$ into the sum
\[A = \sum_{i,j} B^{\prime}_{..}(i,j) \otimes C_{ij}\]
and the canonical diagonal form varies according to the representation

\begin{figure}[h!]
	\centering
	\[D_{1}^{\prime} = \begin{bmatrix}
		\diag(B_{11}(i,i)) & B_{12} & \dots & B_{1a}\\
		B_{21} & \diag(B_{22}(i,i)) & \dots & B_{2a}\\
		\vdots &  & \ddots & \vdots \\
		B_{a1} & B_{a2} & \dots & \diag(B_{aa}(i,i)) 
	\end{bmatrix}\]	
	\caption{The canonical diagonal form $D_{1}^{\prime}$}
	\label{fig: T_{1}a}
\end{figure}		

\begin{figure}[h!]
	\centering
	\[D_{1\frac{1}{2}}^{\prime} = \begin{bmatrix}
		\diag(B_{11}(i,i)) & \diag(B_{12}(i,i)) & \dots & \diag(B_{1a}(i,i))\\
		\diag(B_{21}(i,i)) & \diag(B_{22}(i,i)) & \dots & \diag(B_{2a}(i,i))\\
		\vdots &  & \ddots & \vdots \\
		\diag(B_{a1}(i,i)) & \diag(B_{a2}(i,i)) & \dots & \diag(B_{aa}(i,i)) 
	\end{bmatrix}\]	
	\caption{The canonical diagonal form $D_{1\frac{1}{2}}^{\prime}$}
	\label{fig: T_{1}1/2a}
\end{figure}

However, this dosen't indicate explicitly $B_{ij}$ as the projection of $A$ onto $C_{ij}$ so we still adopt the former $B_{ij}$ as our assumption. \nocite{7555316}\nocite{FOMATATI2022180}\nocite{Lathauwer2008}\nocite{doi:10.1137/07070111X}\nocite{zbMATH04180669}\nocite{zbMATH07021224}

\bibliographystyle{amsplain}
\bibliography{journal}

\appendix

\section{Examples of $4\times 4$ matrices}

This problem originates from decomposition of large covariance matrix with hidden structure in spatio-temporal statistics, wireless communication and mathematical physics etc. For example, the white noise problem states that there is a random signal, viewed as a stochastic process on the Minkovski space or a general manifold with 
\begin{equation}
	\left\langle \eta(\mathbf{x},t)\mid \eta(\mathbf{x}^{\prime},t^{\prime}) \right\rangle = \delta(\mathbf{x} - \mathbf{x}^{\prime})\delta(t - t^{\prime})
	\label{eq: white noise}
\end{equation}
where $\delta$ is the Dirac function. The white noise has the identity matrix as its covariance which can be reduced into the tensor product of two factor matrix
\[I_{qp} = I_{p}\otimes I_{q}\]
where $I$ denotes the identity matrix. This demonstration holds true depends on the hidden structure between $\mathbf{x}$ and $t$. In fact, the white noise of form (\ref{eq: white noise}) has the matrix normal density with respect to two $m \times m$ and $n \times n$ Hermite matrices $A$ and $B$ and a $m\times n$ matrix $M$ 
\[(2\pi)^{-mn/2}|{A} |^{-n/2}|{B} |^{-m/2}\operatorname{etr} \left(-\frac {1}{2} {A} ^{-1}(X - {M} ){B} ^{-1}({X} - {M})^{*}\right)\prod_{i,j}x_{ij}\]
where $\operatorname{etr}(A)$ means $\exp (\tr A)$ and $|A|$ the determinant of $A$, if and only if, $\operatorname{vec}(X)$ viewed as an $mn \times 1$ vector pilling up the columns of the $m \times n$ matrix stacked on top of one another, is the multivariate normal distribution $N_{mn}(\operatorname{vec}(M), A \otimes B)$. The soul problem is to decide whether the covariance of $\operatorname{vec}(X)$, an $mn\times mn$ Hermite positive definite matrix, is of the form $A \otimes B$ for some $m\times m$ and $n \times n$ Hermite matrices $A$ and $B$. However, if the order $n$ is a product of two divisors $a$ and $b$, it's still impossible for all $n \times n$ matrices to possess a factorization of Kronecker product of two $a\times a $ matrix $A$ and $b \times b$ matrix $B$. Let's look at this

\begin{example} The following square matrix of order $4$
	\begin{equation}
		\begin{bmatrix}
			1 & 0.5 & 0.2 & 0.1 \\
			0.5 & 1 & 0.5 & 0.2 \\
			0.2 & 0.5 & 1 & 0.5 \\
			0.1 & 0.2 & 0.5 & 1 
		\end{bmatrix}
		\label{eq: example}
	\end{equation}
	is not a Kronecker product of two $2 \times 2$ matrices.
\end{example} 
Therefore, it sounds absurd to presume such a tensor form since it cannot always describe all situations behind the hidden structure of mathematical models. This initiates our study of weak assumptions for matrices. Although not all square matrices of divisible order $n = ab$ are of the form $A \times B$ for some $a\times a$ matrix $A$ and $b \times b$ matrix $B$, it's still possible for a square matrix to have a reduced form of the sum of Kronecker products. The following three examples show how this could be done.

\begin{example}[Partial diagonalization]
	\[\begin{aligned}
		\begin{bmatrix}
			1 & 0.5 & 0.2 & 0.1 \\
			0.5 & 1 & 0.5 & 0.2 \\
			0.2 & 0.5 & 1 & 0.5 \\
			0.1 & 0.2 & 0.5 & 1 
		\end{bmatrix} &=
		\begin{bmatrix}
			0.5	& -0.1\\
			-0.1 & 0.5 
		\end{bmatrix} \otimes 
		\begin{bmatrix}
			0.5	& -0.5\\
			-0.5 & 0.5
		\end{bmatrix} + 
		\begin{bmatrix}
			0.5	& 0.2\\
			0.2 & 1.5 
		\end{bmatrix} \otimes 
		\begin{bmatrix}
			-0.5 & -0.5\\
			0.5 & 0.5
		\end{bmatrix} \\
		& +  
		\begin{bmatrix}
			1.5	& -0.2\\
			-0.2 & 0.5 
		\end{bmatrix} \otimes 
		\begin{bmatrix}
			-0.5 & 0.5\\
			-0.5 & 0.5
		\end{bmatrix} 
		+ 
		\begin{bmatrix}
			1.5	& 0.5\\
			0.5 & 1.5 
		\end{bmatrix} \otimes 
		\begin{bmatrix}
			0.5 & 0.5\\
			0.5 & 0.5
		\end{bmatrix}\\ \\
		&= A_{11} \otimes B_{11} + A_{12} \otimes B_{12} + A_{21} \otimes B_{21} + A_{22} \otimes B_{22}
	\end{aligned}
	\]
	where $\rank(B_{11}) = \rank(B_{12}) = \rank(B_{21}) = \rank(B_{22}) = 1$.
\end{example}

\begin{example}[Partially orthogonal diagonalization]
	\[\begin{aligned}
		\begin{bmatrix}
			1 & 0.5 & 0.2 & 0.3 \\
			0.5 & 1 & 0.3 & 0.2 \\
			0.2 & 0.3 & 1 & 0.5 \\
			0.3 & 0.2 & 0.5 & 1 
		\end{bmatrix} &=
		\begin{bmatrix}
			0.5	& -0.1\\
			-0.1 & 0.5 
		\end{bmatrix} \otimes 
		\begin{bmatrix}
			0.5	& -0.5\\
			-0.5 & 0.5
		\end{bmatrix} 
		+ 
		\begin{bmatrix}
			1.5	& 0.5\\
			0.5 & 1.5 
		\end{bmatrix} \otimes 
		\begin{bmatrix}
			0.5 & 0.5\\
			0.5 & 0.5
		\end{bmatrix}\\
		&= A_{1} \otimes B_{1} + A_{2} \otimes B_{2}
	\end{aligned}
	\]
	where $\rank(B_{1}) = \rank(B_{2}) = 1$ and $B_{1}$ and $B_{2}$ are commutative.
\end{example}

\begin{example}[Totally orthogonal diagonalization]
	\[\begin{aligned}
		\begin{bmatrix}
			1 & 0.15 & 0.15 & 0.5 \\
			0.15 & 1 & 0.5 & 0.15 \\
			0.15 & 0.5 & 1 & 0.15 \\
			0.5 & 0.15 & 0.15 & 1 
		\end{bmatrix} &=
		\frac{6}{5}\begin{bmatrix}
			0.5	& -0.5\\
			-0.5 & 0.5
		\end{bmatrix} \otimes 
		\begin{bmatrix}
			0.5	& -0.5\\
			-0.5 & 0.5
		\end{bmatrix} 
		+ \frac{1}{2} 
		\begin{bmatrix}
			0.5	& -0.5\\
			-0.5 & 0.5
		\end{bmatrix} \otimes 
		\begin{bmatrix}
			0.5 & 0.5\\
			0.5 & 0.5
		\end{bmatrix}\\ 
		& + \frac{1}{2}\begin{bmatrix}
			0.5 & 0.5\\
			0.5 & 0.5
		\end{bmatrix} \otimes 
		\begin{bmatrix}
			0.5	& -0.5\\
			-0.5 & 0.5
		\end{bmatrix} 
		+ \frac{6}{5}  
		\begin{bmatrix}
			0.5 & 0.5\\
			0.5 & 0.5
		\end{bmatrix} \otimes 
		\begin{bmatrix}
			0.5 & 0.5\\
			0.5 & 0.5
		\end{bmatrix}\\ 	\\
		&= c_{11} A_{1} \otimes B_{1} + c_{12} A_{1} \otimes B_{2} + c_{21} A_{2} \otimes B_{1} + c_{22} A_{2} \otimes B_{2}
	\end{aligned}
	\]
	where $\rank(A_{1}) = \rank(A_{2}) = \rank(B_{1}) = \rank(B_{2}) = 1$, $A_{1}$ and $A_{2}$ are commutative and $B_{1}$ and $B_{2}$ are commutative.
\end{example}

\section{Factorization algorithm}

\setcounter{algorithm}{-1}
\begin{algorithm}[htbp]
	\caption{An algorithm for tensor decomposition}
	\begin{algorithmic}
		\Require An $n \times n$ Hermite matrix $A$
		\Ensure $n = p_{1}^{r_{1}} p_{2}^{r_{2}}\cdots p_{t}^{r_{t}}$ with $p_{1},\dots, p_{t}$ distinct primes and $r_{1},\dots, r_{t} \geq 0 $
		\State $m = n$ 
		\State $s \gets 0$ \Comment{Set global variable}
		\While{$m \geq 4$} \Comment{Main algorithm}
		\For{$i = 1: r$}
		\For{$j=1:r[i]$}
		\State $q = p[j]$
		\State $m = n/q$
		\State Run Algorithm 1 with input $A,n,q$
		\EndFor
		\EndFor
		\EndWhile
	\end{algorithmic}
\end{algorithm}

\begin{algorithm}[htbp]
	\caption{An algorithm for $T_{1}$}
	\begin{algorithmic}
		\Require An $n \times n$ Hermite matrix $A$
		\Ensure A natural number $n$ and a prime number $p$ such that $p|n$
		\State $m = n/p$
		\For{$i = 1: m$}
		\For{$j = 1:m$}
		\State $A[i,j] \gets A[(ip+1):(i+1)p,(jp+1):(j+1)p]$
		\EndFor
		\EndFor
		\For{$i,j = 1: m$}
		\If{$A[i,i]A[j,j] = A[j,j] A[i,i] $}
		\State Find the eigenvectors $C \gets [c[1], \dots, c[p]]$ for $A[i,i]$  
		\Else 
		\State $s \gets [s,0]$ \Comment{$T_{1}$ failure!}
		\State \textbf{quit} 
		\EndIf
		\EndFor
		\For{$i = 1: m$}
		\For{$k,l = 1: p$}
		\State $B[i,i][k,l] \gets c[i]^{*}A[k,l]c[i]$		
		\EndFor
		\EndFor
		\State $s \gets [s,1]$ \Comment{$T_{1}$ success!}
		\State Run Algorithm 2 with input $A,n,p$\\
		\Return $B[i,j],C,s$
		
	\end{algorithmic}
\end{algorithm}

\begin{algorithm}[htbp]
	\caption{An algorithm for $T_{1\frac{1}{2}}$}
	\begin{algorithmic}
		\Require An $n \times n$ Hermite matrix $A$
		\Ensure A natural number $n$ and a prime number $p$ such that $p|n$
		\State $m = n/p$
		\For{$i = 1: m$}
		\For{$j = 1:m$}
		\State $A[i,j] \gets A[(ip+1):(i+1)p,(jp+1):(j+1)p]$
		\EndFor
		\EndFor
		\For{$i1,i2 = 1: m$}
		\For{$j1 = i1: m$}
		\For{$j2=i2:m$}
		\If{$A[i1,j1]A[i2,j2] = A[i2,j2] A[i1,j1] $}
		\State Find the eigenvectors $C \gets [c[1], \dots, c[p]]$ for $A[i1,j1]$  
		\Else 
		\State $s \gets [s,0]$ \Comment{$T_{1\frac{1}{2}}$ failure!}
		\State \textbf{quit} 
		\EndIf
		\EndFor
		\EndFor
		\EndFor
		\For{$i = 1: m$}
		\For{$k,l = 1: p$}
		\State $B[i,i][k,l] \gets c[i]^{*}A[k,l]c[i]$		
		\EndFor
		\EndFor
		\State $s \gets [s,2]$ \Comment{$T_{1\frac{1}{2}}$ success!}
		\State Run Algorithm 3 with input $B[i,i],m,p$\\
		\Return $B[i,j],C,s$
	\end{algorithmic}
\end{algorithm}

\begin{algorithm}[htbp]
	\caption{An algorithm for $T_{2}$ and $T_{3}$}
	\begin{algorithmic}
		\Require A cell of $p \times p$ Hermite matrices $B[i]$
		\Ensure A natural number $n$ and a prime number $p$
		\State $a \gets \text{zeros}[p,m]$
		\For{$i = 1: n$}
		\If{$B[i]B[j] = B[j] B[i] $}
		\State Find the eigenvectors $C \gets [c[1], \dots, c[p]]$ and 
		\State eigenvalues $a[:,i] \gets [a[1],\dots, a[p]] $ for $B[i]$  
		\Else 
		\State $s \gets [s,0]$ \Comment{$T_{1\frac{1}{2}}$ failure!}
		\State \textbf{quit} 
		\EndIf
		\EndFor
		\State $s \gets [s,3]$ \Comment{$T_{2}$ success!}
		\State Compute the similar diagonal form $d \gets z^{-1}az$
		\If{$\text{nonzeros}(d) \geq 1$}
		\State $s \gets [s,0]$ \Comment{$T_{3}$ failure!}
		\State \Return $B[i,j],C,s$		
		\Else		
		\State Compute $x \gets A[1,:]$
		\State Solve $y \gets A x$ 
		\State $s \gets [s,4]$ \Comment{$T_{3}$ success!}
		\EndIf
		\State \Return $x,y$		
	\end{algorithmic}
\end{algorithm}

\end{document}